\documentclass[11pt]{amsart}
\usepackage{geometry,color,amsmath,amsthm,amssymb,amsfonts,graphicx,setspace,verbatim,float}
\usepackage[all,color]{xy}
\usepackage[pagewise]{lineno}
\makeatother  
\makeatletter   

\newcommand{\Rnum}[1]{\expandafter\@slowromancap\romannumeral #1@}

\newcommand{\di}{\mathcal{DI}}
\newcommand{\tr}{\mathcal{TR}}
\newcommand{\te}{\mathcal{TE}}
\newcommand{\he}{\mathcal{HE}}

\newtheorem{theorem}{Theorem}[section]
\newtheorem{lemma}[theorem]{Lemma}
\theoremstyle{definition}

\newtheorem{proposition}[theorem]{Proposition}

\newtheorem{remark}[theorem]{Remark}

\title{Cubic tessellations of the helicosms}
\author{Isabel Hubard, Mark Mixer, Daniel Pellicer, Asia Ivi\'c Weiss}

\begin{document}
\begin{abstract}
Up to isomorphism there are six fixed-point free crystallographic groups in Euclidean Space generated by twists (screw motions).
In each case, an orientable 3-manifold is obtained as the quotient of $\mathbb{E}^3$ by such a group.
The cubic tessellation of $\mathbb{E}^3$ induces tessellations on each such manifold.  These tessellations of the 3-torus and the didicosm were classified as `equivelar toroids' and `cubic tessellation of the didicosm' in previous works.
This paper concludes the classification of cubic tessellations on the remaining four orientable manifolds.
\end{abstract}

\maketitle

\section{Introduction}

Among the seventeen crystallographic groups of the Euclidean plane, there are only two that act with no fixed-points.  Identifying points in the same orbit of the action of these groups leads to the well-known flat, closed 2-manifolds which are commonly known as the torus and the Klein bottle.  The three regular tessellations of the Euclidean plane, in a natural way, give rise to the equivelar tessellations of these surfaces.  These tessellations have been of interest to many (see for example \cite{Edmonds, Kurth, Thomassen}), and have a long history which includes the classic work of Coxeter \cite{Coxeter48} (see also \cite{CM}), giving the classification of regular and chiral maps on the torus.   The topic has also drawn more recent attention, with the works of Brehm, K\"uhnel \cite{BK}, and of Wilson \cite{unclesteve}, classifying all equivelar maps on the torus and Klein bottle, respectively.

There are precisely ten flat, closed 3-manifolds (see \cite{HW, Nowacki}).  Each manifold arises as the quotient of $\mathbb{E}^3$ by a fixed-point free crystallographic group, and thus can be seen in the broader setting of space-forms.  In fact, the Euclidean space-forms are the complete connected flat riemannian manifolds; for more details on this topic see \cite{ARP} and \cite{Wolf}.

Various properties of the flat, closed 3-manifolds have been studied recently (see for example \cite{CR2}, \cite{DR}, \cite{Isangulov}).  In \cite{CR1} Conway and Rossetti coin the name \textit{platycosms} for these manifolds, and give them a set of individual names. Six of these manifolds are orientable and they can be grouped into platycosms satisfying one of three properties. The 3-torus is the most symmetric platycosm in the sense that it has an infinite group of isometries of the manifold containing $S^1 \times S^1 \times S^1$ as a subgroup. As a consequence of this, it is the only platycosm admitting regular tessellations \cite{HMS}. In contrast, the isometry group of the didicosm (also known as the Hantzsche-Wendt manifold \cite{Wolf}) is finite and in this sense it is the least symmetric of the platicosms.

In this paper we focus on the remaining four orientable platycosm which, together with the torus, are named the helicosms. Their isometry groups contain a finite index subgroup isomorphic to $S^1$, which induces a distinguished direction.

Our interest centers on the situation when the isometry group $\mathcal{G}$ inducing a helicosm as a quotient, is taken to be a subgroup of the symmetry group $[4,3,4]$ of the regular tessellation $\mathcal{U}$ of $\mathbb{E}^3$ by cubes. The orbit-space $\mathcal{U}/\mathcal{G}$ gives rise to what we call a cubic tessellation of the corresponding helicosm.  That is to say, for each $i \in \{0,1,2,3\}$, the orbit of an $i$-dimensional face of the tessellation of $\mathbb{E}^3$ yields an $i$-dimensional face of the tessellation of the helicosm.  In non-degenerate cases, one can think of the combinatorial structure of this tessellation as an abstract polytope with Schl\"afli type $\{4,3,4\}$ (see \cite{ARP}).
Given any tessellation $\Lambda$ of $\mathbb{E}^n$, and a fixed-point free discrete group $\mathcal{G}$ of Euclidean isometries, if $\mathcal{G}$ preserves the tessellation, then we call the orbit-space of $\Lambda /\mathcal{G}$ a \textit{twistoid} on the manifold $\mathbb{E}^n/\mathcal{G}$.

The particular instance of twistoids on the $n$-torus has been previously considered, where they are called \textit{toroids}.  Regular tessellations of these manifolds were classified by McMullen and Schulte \cite{MSHigherToroidal}.  Subsequently, McMullen proved there are no chiral tessellations of the $n$-torus (for $n \geq 3$) (see \cite[Section 6H]{ARP}).  Recently the classification by symmetry type of all cubic tessellations of the 3-torus and didicosm were completed by Hubard, Orbani\'{c}, Pellicer, and Weiss \cite{TOROIDS} and Hubard, Mixer, Pellicer and Weiss \cite{TWISTOIDS1}, respectively.

This paper is the third in a sequence of four papers classifying cubic tessellations of the platycosms,  the first two  being \cite{TOROIDS} and \cite{TWISTOIDS1}.  Sections~\ref{sectionheli}, \ref{s:twistoids} and~\ref{s:orientable} of this paper  provide introductory concepts for the helicosms and their twistoids.
In Section~\ref{s:orientable} we show that there are no cubic tessellations on the hexacosm and in Sections~\ref{s:di}, \ref{s:tri} and ~\ref{s:tetra} we classify all possible cubic tessellations of the dicosm, tricosm, and tetracosm, respectively, according to the structure of their automorphism group.
 Finally, in Section~\ref{s:toroids} we relate the contents of this paper with the first one of the series.
In the final paper, we complete the classification by considering the four non-orientable manifolds.  Our approach is to use the geometry of three dimensional Euclidean space in order to understand how the fixed-point free crystallographic groups interact with the symmetries of the cubic tessellation.

\section{Flat 3-manifolds}
\label{sectionheli}

According to the notation of \cite[Chapter 24]{ConwaySOT}, the six orientable compact flat quotients of $\mathbb{E}^3$ are the didicosm and the five {\em helicosms} called torocosm (3-torus), dicosm, tricosm, tetracosm and hexacosm.
The cubical tessellations of the didicosm and the torocosm were respectively studied in \cite{TWISTOIDS1} and \cite{TOROIDS}, respectively.
In this section, we analyze the structure of the groups
   $\di := 2222 \frac{1}{2} + \frac{1}{2} + \frac{1}{2} + \frac{1}{2} +$,
   $\tr := 333 \frac{1}{3} + \frac{1}{3} + \frac{1}{3} +$,
   $\te := 442 \frac{1}{4} + \frac{1}{4} + \frac{1}{2} +$, and
   $\he := 632 \frac{1}{6} + \frac{1}{3} + \frac{1}{2} +$,
as well as the corresponding quotients of the Euclidean space yielding the dicosm, tricosm, tetracosm and hexacosm, respectively.

For each group $\mathcal G \in \{\di,\tr,\te,\he\}$, a natural set of generators for $\mathcal G$ consists of isometries called \textit{twists}, where a twist (or screw motion) is the commuting product of a rotation (of order $n \ge 2$) about a line with directional vector $v$ and the translation by $v$. We call $v$ the \textit{translational component} of the twist and say that it has \textit{rotational component of order} $n$.  If the rotational component is of order $n$, we say that the twist is an $n$-turn twist. We shall also refer to the 2-turn twists by \textit{half-turn twists}.

In all cases the groups
are generated by twists with parallel axes. The group $\di$ is generated by four half-turn twists and the group $\tr$ is generated by three $3$-turn twists. The group $\tr$ is generated by two $4$-turn twists and a half-turn twist. Finally, the group $\he$ is generated by a $6$-turn twist, a $3$-turn twist and a half-turn twist.
The intersection of the axes generating $\di$, $\tr$, $\te$, or $\he$ with a perpendicular plane form, respectively, a parallelogram, an equilateral triangle, an isoceles triangle with a straight angle, or a triangle with angles $\pi/6$, $\pi/3$ and $\pi/2$.

The translational component of all generating twists of $\di$ and of $\tr$ coincide. The translational component of the $4$-turn twists is half of the translational component of the half-turn twist generating $\te$. Finally, considering the generators of $\he$, the translational component of the $6$-turn twist is half the translational component of the $3$-turn twist and a third of the translational component of the half-turn twist.  To conclude this section we describe a fundamental region for each of the dicosm, tricosm, tetracosm and hexacosm.

\begin{figure}
\begin{center}
\includegraphics[width=7cm, height=5cm]{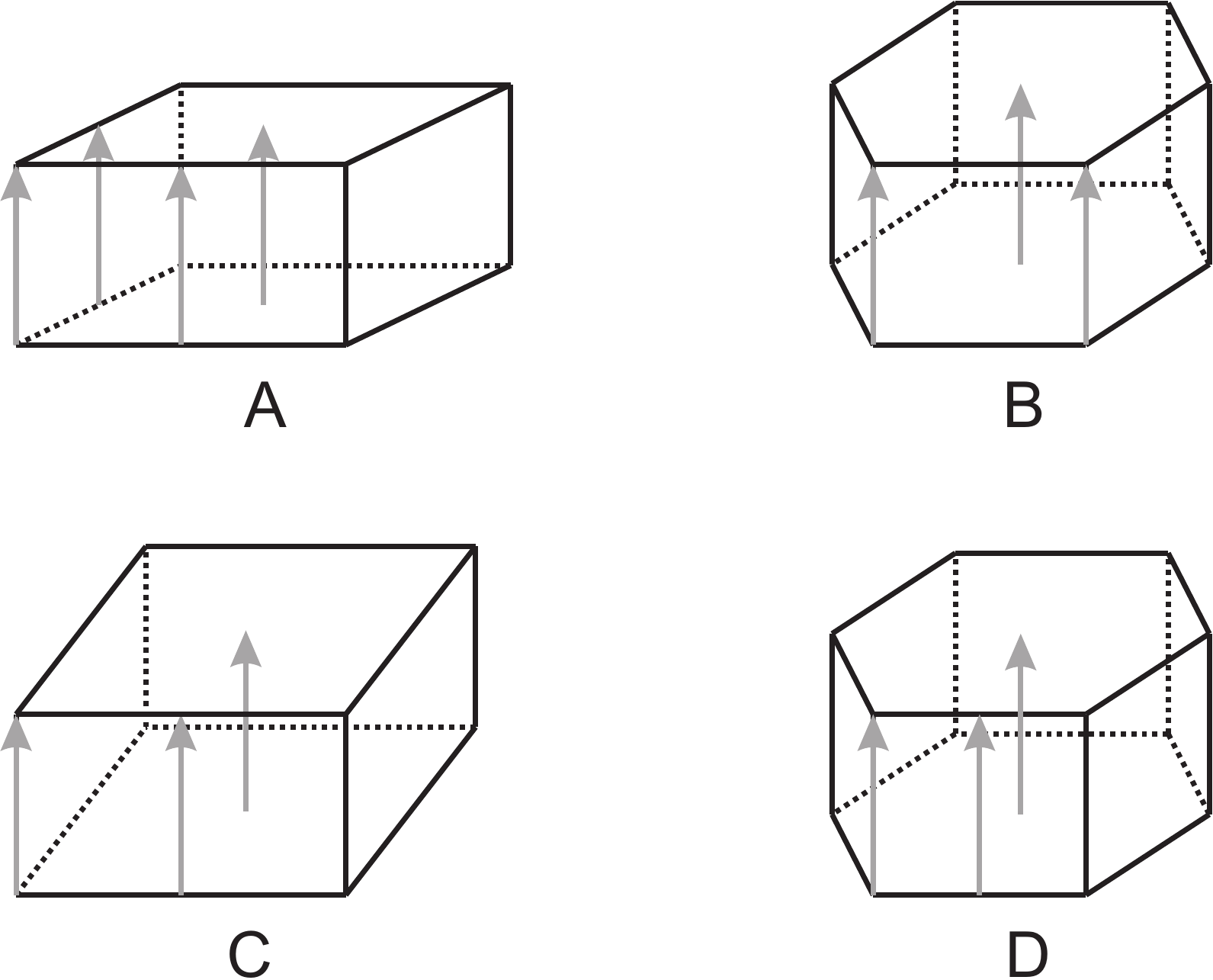}
\caption{Fundamental regions of the non-toroidal helicosms \label{fig:FundReg}}
\end{center}
\end{figure}

Figure~\ref{fig:FundReg} (A), (B), (C) and (D) shows a fundamental region for the dicosm, tricosm, tetracosm and hexacosm, respectively.
In all cases the fundamental region is an upright prism over a centrally symmetric polygon (parallelogram, regular hexagon, square and regular hexagon). This polygon is in a plane perpendicular to the direction of the translational component of the generating twists, indicated by gray arrows. The hight of the prism is the norm of the smallest translational component of the generating twists.

The axes intersecting the midpoint of edges of the prisms in Figure~\ref{fig:FundReg} (C) and (D) correspond to the generating half-turn twists. The axis containing a vertex of the prism in Figure~\ref{fig:FundReg} (D) corresponds to the generating 3-turn twist.

The identification of the fundamental regions to obtain the helicosms is as follows. Each vertical wall of the prism is identified with the opposite wall by translation; note that in all cases, a vertical wall is in a plane parallel to that of the opposite wall.
The top and bottom of the prism are identified by the twist with its axis lying in the middle; that is, the top and bottom of the prisms in Figure~\ref{fig:FundReg} (A), (B), (C) and (D) are identified by a half-turn, 3-turn, 4-turn and 6-turn twist, respectively.


\section{Twistoids and symmetries of the cubic tessellation}
\label{s:twistoids}

Up to similarity, there is a unique regular tessellation $\mathcal{U}$ of Euclidean 3-space $\mathbb{E}^3$ by cubes.   We assume that the vertices of this tessellation coincide with the integer lattice $\mathbb{Z}^3$ determined by the standard basis $\{e_1, e_2, e_3 \}$.  The group of symmetries $\Gamma(\mathcal{U})$ of this tessellation is the Coxeter group $[4,3,4]$, and $\Gamma(\mathcal{U}) \cong {\bf T} \rtimes {\bf S}$ where ${\bf T}$ is the group of all translational symmetries, and $\bf{S}$ is the stabilizer of the origin.  The group ${\bf S}$ is called the \textit{point group} of $\mathcal{U}$.

Let $\mathcal{G}$ be a fixed point free discrete group of isometries of $\mathbb{E}^3$.  When $\mathcal{G}$ is a subgroup of the symmetry group $\Gamma(\mathcal{U})$, the quotient $\mathcal{U} / \mathcal{G}$ is called a \textit{(cubic) twistoid}. Hence, a cubic twistoid can be seen as a tessellation by cubes of the platycosm associated with $\mathcal{G}$.  The twistoids  $\mathcal{U} / \mathcal{G}$ and $\mathcal{U} / \mathcal{G}'$ are said to be {\em isomorphic} if $\mathcal{G}$ and $\mathcal{G}'$ are conjugates in $\Gamma(\mathcal{U})$.  The definition of isomorphism allows us to relabel any integral point $(x,y,z)$ as $(0,0,0)$ while not changing the isomorphism class of a twistoid.
Similarly, permuting the $x$, $y,$ and $z$ axes as well as switching the positive and negative direction of any coordinate axis, does not change the isomorphism class of a twistoid.  Later, we shall make use of this freedom in the position and orientation of the integer lattice in order to define simple parameters that classify the twistoids on the helicosms.

Consider the translation $\tau$ of $\mathbb{E}^3$ by the vector $(1/2, 1/2, 1/2)$.
The {\em dual} of a twistoid $\mathcal T = \mathcal{U} / \mathcal{G}$ is the twistoid $\mathcal{T}^*= \mathcal{U} / \mathcal{G}^\tau$ (where $\mathcal{G}^\tau$ denotes the conjugated of $\mathcal G$ by $\tau$).
Equivalently, one can obtain $\mathcal{T}^*$ as the twistoid $\mathcal{U}^*/\mathcal{G}$, where $\mathcal{U}^*$ denotes the geometrically dual tessellation of $\mathcal{U}$ in the usual sense. Hence, given a cubic twistoid $\mathcal T$, its dual $\mathcal{T}^*$ is the cubic tessellation of the platycosm whose vertices, midpoints of edges and squares are the midpoints of the cubes, squares and edges, respectively, of the twistoid $\mathcal T$. Thus, the midpoints of the cubes of $\mathcal{T}^*$ correspond to the vertices of $\mathcal T$.

In \cite[Section 2]{TWISTOIDS1} we described all twists occurring as symmetries of the cubic tessellation $\mathcal{U}$.
The order of the rotational component of any such twists is either $2$, $3$ or $4$. Figure~\ref{figtwists1} and Table~\ref{twists} illustrate the eleven kinds of twists that are symmetries of $\mathcal{U}$.
In Table~\ref{twists}, for each type,
an $\times$ in the columns labeled `V',`E',`S',`C' indicates whether  the axis of a twist intersects the centroid of a vertex, edge, square, or cube, respectively.
Additionally, the column labeled `Direction' gives one possible vector for the direction of the twist, with all other possible vectors being an image of this vector under the point group \textbf{S}. The column labeled `Norm' indicates the numbers that occur as norms of the translational component of the twist.

\begin{figure}
$$\includegraphics[scale=.25]{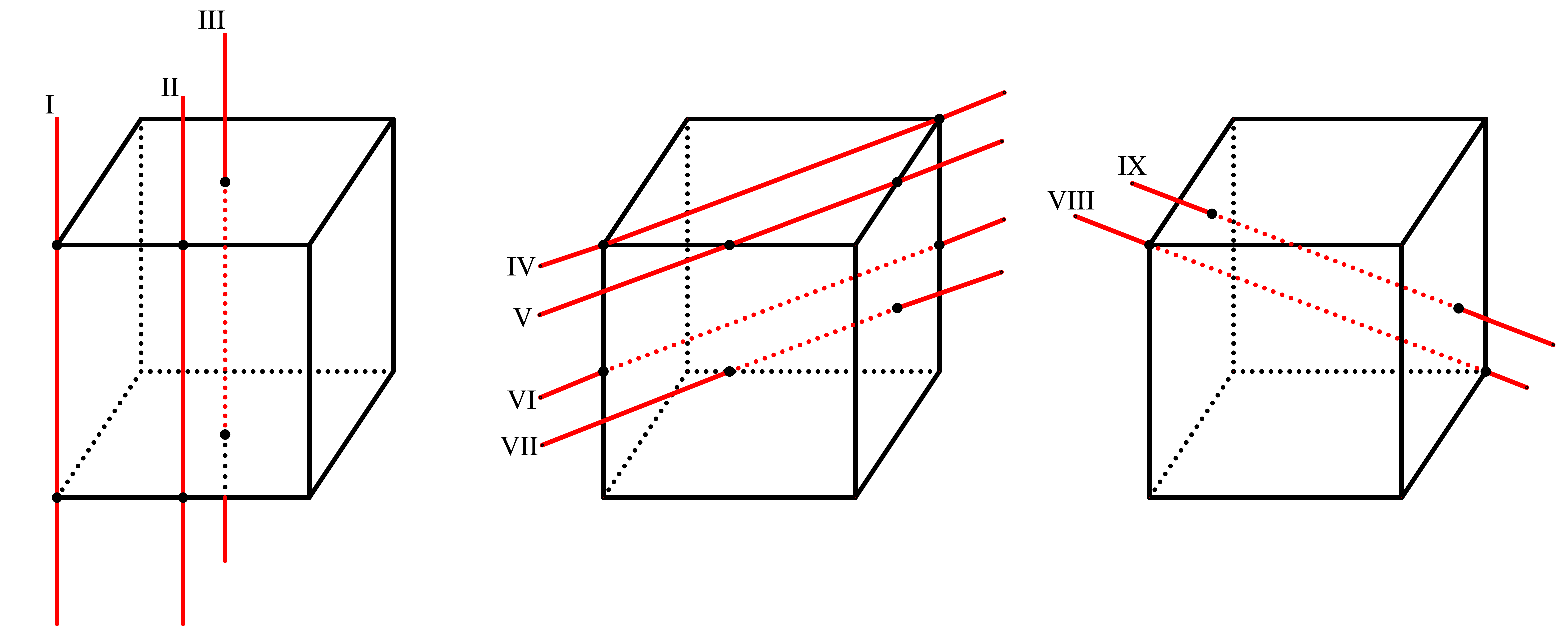}$$
\caption{Twist Symmetries of the cubic tessellation \label{figtwists1}}
\end{figure}

\begin{table}[H]
\begin{tabular}{|c|c|c|c|c|c|c|c|c|} \hline
Type & Period  & V & E & S & C &Direction & Norm \\ \hline
\Rnum{1} & 2   & $\times$ & $\times$ & & & $e_1$&$\mathbb{Z}$ \\ \hline
\Rnum{2} & 2   &  & $\times$&$\times$ & &$e_1$ &$\mathbb{Z}$ \\ \hline
\Rnum{3} & 2   &  &  &$\times$ &$\times$ & $e_1$&$\mathbb{Z}$ \\ \hline
\Rnum{4} & 2   & $\times$ &  &$\times$ & & $e_1+e_2$&$\sqrt{2} \mathbb{Z} $ \\ \hline
\Rnum{5} & 2   & & $\times$ & & & $e_1+e_2$& $\frac{\sqrt{2}}{2} \mathbb{Z} \setminus \sqrt{2} \mathbb{Z} $ \\ \hline
\Rnum{6} & 2   &  & $\times$ & &$\times$ & $e_1+e_2$&$ \sqrt{2} \mathbb{Z}$ \\ \hline
\Rnum{7} & 2   &  &  &$\times$ & & $e_1+e_2$ &$ \frac{\sqrt{2}}{2} \mathbb{Z} \setminus \sqrt{2} \mathbb{Z}$ \\ \hline
\Rnum{8} & 3   & $\times$ & & &  $\times$& $e_1+e_2+e_3$ & $ \sqrt{3} \mathbb{Z}$ \\ \hline
\Rnum{9} & 3  &  & & & & $e_1+e_2+e_3$ & $ \frac{\sqrt{3}}{3} \mathbb{Z} \setminus  \sqrt{3} \mathbb{Z} $\\ \hline
\Rnum{10} & 4  & $\times$ & $\times$ & & &  $e_1$&$\mathbb{Z}$ \\ \hline
\Rnum{11} & 4   &  &  &$\times$ &$\times$ &  $e_1$&$\mathbb{Z}$ \\ \hline
\end{tabular}
\caption{The conjugacy classes of twists in the group [4,3,4] \label{twists}}
\end{table}

A {\em Petrie polygon} of $\mathcal{U}$ is an infinite path where
 three, but not four, consecutive edges belong to a cube of $\mathcal{U}$. They are helices over triangles, living around axes of type $IX$ (see Figure~\ref{figtwists1}). The images of the Petrie polygon with vertices
\[\{(k,k,k), (k,k+1,k), (k+1,k+1,k) \mid k\in \mathbb{Z}\}\]
under the rotations preserving $\mathcal{U}$ are called {\em right Petrie paths}. The remaining Petrie polygons are called {\em left Petrie paths}, for example the one with vertex set
\[\{(k,k,k),(k+1,k,k),(k+1,k+1,k) \mid k\in \mathbb{Z}\}.\]

\section{Symmetries of the helicosms}
\label{s:orientable}
In this section we give definitions of symmetries and automorphisms of the platycosms.  As we are interested in manifolds endowed with the metric inherited from the cubic tessellation, our definitions differ slightly from some found in other sources which emphasize a topological point of view.

Let $\mathcal{M}$ be the helicosm defined by the group $\mathcal{G}$, and let $\mathcal{T}$ be a twistoid on $\mathcal{M}$.  The largest group of symmetries of $\mathcal{G}$ in $\mathbb{E}^3$ is its {\em affine normalizer}, that is the group $Aff(\mathcal{G})$ of affine transformations of $\mathbb{E}^3$ that fix $\mathcal{G}$ by conjugation.
However, some of the elements of the affine normalizer of $\mathcal G$ 
do not preserve the structure of $\mathcal M$.
Therefore, the largest group of symmetries of $\mathcal G$ that we consider in this paper is the subgroup of $Aff(\mathcal{G})$ defined by (rigid) isometries of $\mathbb{E}^3$ which preserve the cubic tessellation.
Hence, as we did in~\cite{TOROIDS} and~\cite{TWISTOIDS1}, we define the {\em symmetry group} of the twistoid $Sym(\mathcal{T})$ precisely as the subgroup of $Aff(\mathcal{G})$ consisting of (rigid) isometries of $\mathbb{E}^3$ which preserve $\mathcal{G}$ by conjugation.

Following the ideas in~\cite{CR1}, we break $Aff(\mathcal{G})$ - and thus $Sym(\mathcal{T})$ - into `parts'.
The first part (in both groups) will be called the {\em component of identity}.
In $Aff(\mathcal{G})$ these are the affinities of $\mathbb{E}^3$ that fix the group pointwise;
while the component of identity of $Sym(\mathcal{T})$ consists of the isometries of $\mathbb{E}^3$ that fix the group pointwise and preserve the cubic tessellation.  We note that for the helicosms, the component of identity (of both $Aff(\mathcal{G})$ and $Sym(\mathcal{T})$) consists only of translations.

It is known (see for example~\cite{Charlap}) that factoring $Aff(\mathcal{G})$ by its component of identity yields the group $Aut(\mathcal{G})$ of abstract automorphisms of $\mathcal{G}$.
Since $\mathcal{G}$ has a trivial center, we know that $\mathcal G$ is isomorphic to the group $Inn({\mathcal G})$ of inner automorphisms, and therefore, in order to understand the symmetries of a twistoid on the corresponding helicosm, we need to analyze the isometries of $\mathbb{E}^3$ which act like outer automorphisms of $\mathcal{G}$ and preserve the cubic lattice.    We call the collection of such isometries the {\em outer part} of $Sym(\mathcal{T})$.
This collection clearly does not form a group; however, when taking the quotient of $Sym(\mathcal{T})$ by the component of identity and then by $\mathcal{G}$ we obtain a quotient of $Aut(\mathcal{G})$.
We will call this by the group of {\em outer symmetries} of $\mathcal{T}$, and denoted by $Out({\mathcal T})$.

In~\cite{CR1}, the outer automorphism group - and thus our group of outer symmetries of a twistoid - is broken down further.   We use this idea to break the outer part of $Aff(\mathcal{G})$ and $Sym(\mathcal{T})$ further as well.   The {\em rigidly isometric part} comes from symmetries which are isometries for all parameters defining the fundamental region, whereas the {\em deformable part} comes from affinities that might be isometries for specific parameters.
As above, these collections are not groups.  However, the rigidly isometric part is a normal subgroup of the group of outer symmetries of $\mathcal{T}$.
When taking quotient by this normal subgroup we get a group which we will call the group of {\em deformable symmetries} of $\mathcal{T}$.
The structure of the outer automorphism groups of $\di$, $\tr$, $\te$, and $\he$ are considered in both~\cite{CR1} and~\cite{Hillman}.  We summarize their results below.

\begin{itemize}
\item $Out(\di) \cong (\mathbb{Z}_2^2 \rtimes PGL(2,\mathbb{Z})) \times \mathbb{Z}_2$
\item $Out(\tr) \cong S_3 \times \mathbb{Z}_2$
\item $Out(\te) \cong \mathbb{Z}_2^2$
\item $Out(\he) \cong \mathbb{Z}_2$
\end{itemize}


As in~\cite{TOROIDS} and \cite{TWISTOIDS1},
we define the automorphism group $Aut(\mathcal{T})$ for a twistoid $\mathcal{T}$ as $Sym(\mathcal{T})/\mathcal{G}$.
Furthermore, we denote the group homomorphism that sends $Sym(\mathcal{T})$ to $Aut(\mathcal{T})$ by $\phi$, and say that a symmetry $\alpha$ {\em induces} an automorphism $\overline{\alpha}$ when $\phi(\alpha) = \overline{\alpha}$.   Using the language above, we will also use the notions of the component of identity of $Aut(\mathcal{T})$, as well as the group of rigidly isometric automorphisms, and the group of deformable automorphisms.  Based on the discussion above, in order to classify twistoids on the helicosms by symmetry type we will show how both the component of identity and the outer part of $Sym(\mathcal{T})$ can be determined by the parameters defining a twistoid $\mathcal{T}$.

We start by ruling out the possibility of cubic twistoids on the hexacosm.

\begin{proposition}
There are no cubic twistoids on the hexacosm.
\end{proposition}

\begin{proof}
If there were a cubic twistoid on that manifold, then $\he$ could be realized as a subgroup of the group $[4,3,4]$ of isometries of $\mathbb{E}^3$ that preserve a cubic tessellation.  However, as noted in Table~\ref{twists}, the group $[4,3,4]$ does not contain any 6-fold twists, whereas $\he$ does.
\end{proof}

The following sections give the classification, by symmetry type, of the cubic twistoids on the dicosm, tricosm and tetracosm.

\section{Dicosm}\label{s:di}

In this section we will show that the twistoids on dicosm can be classified into twenty-two distinct families (we brake the analysis in two cases resulting with eighteen families in one and four in the other).

It follows from the notation of~\cite{ConwaySOT} that the group $\mathcal{DI}$ is generated by four twists whose axes are parallel and project orthogonally to the vertices of a parallelogram.
Furthermore, the translation component of each twist coincides. It can be easily seen that one of the twists can be obtained as a composition the other three.
Hence, the group $\mathcal{DI}$ is generated by three twists whose axes are parallel - but not coplanar - and whose translation component is equal.

The analysis naturally splits in two cases. One of them is when the axes of the twists are parallel to a coordinate axis (that is, they are of type $I$, $II$ or $III$); without loss of generality, we take the axes to be in the direction of $(0,0,1)$.  In the other case the axes are parallel to a diagonal of a square of the cubic tessellation (that is, they are of types $IV$, $V$, $VI$ or $VII$); without loss of generality, we take the axis to be in the direction of $(1,-1,0)$.  In both cases, the component of identity of any twistoid on the dicosm consists of all translations preserving the cubic tessellation $\mathcal{U}$, in the direction of the axes of the generating twists.  The following lemma will be useful in finding a set of simple parameters to classify twistoids on the dicosm.

\begin{lemma}\label{q2}
Let $\Sigma$ be a group generated by two linearly independent translations in $\mathbb{E}^2$.
\begin{enumerate}
\item  The group $\Sigma$ is generated by translations $(a,b)$ and $(c,0)$ for some $a$, $b$, and $c$.
\item The parameters $a$, $b$, and $c$ can be chosen so that $0 \leq a < c$.
\end{enumerate}
\end{lemma}

\begin{proof}
The first item is a direct consequence of Lemma 4 of~\cite{TOROIDS}.   The second item follows from the fact that the translation by the vector $n(c,0)$ is in $\Sigma$ for any integer $n$.
\end{proof}

\medskip
{\bf Case 1:} The axes of the generators have direction vector $(0,0,1)$.
\medskip

Let $c \in \mathbb{Z}$ be the norm of the translational component of the generating twists. If the axis of one the generators contains a vertex of $\mathcal{U}$, then we denote that generator by $\sigma_1$. Furthermore, as we are considering twistoids up to duality, if none of the axes contains a vertex of $\mathcal{U}$, but one of them contains a centroid of a cube, we may consider the corresponding twist as $\sigma_1$ and work with the dual tessellation.

Let $\Pi$ be the $xy$-plane. Much of our analysis will be done by considering how the axes of the generating twists in the cubic tessellation project into $\Pi$.
Assume that the axes of $\sigma_1$, $\sigma_2$ and $\sigma_3$ intersect $\Pi$ on the points $(p_1,q_1,0)$, $(p_2,q_2,0)$, and $(p_3,q_3,0)$, respectively. Up to duality and because of the freedom on the choice of the $x$ and $y$ axes, we may assume that $q_1=q_2=0$,  $p_1 \in \{0,\frac{1}{2}\}$, $p_2>p_1$ and $p_3 \ge 0$, $q_3>0$.  Note that $p_2, p_3, q_3 \in \frac{1}{2}\mathbb{Z}$. As a consequence of our previous assumptions, if $p_1 = \frac{1}{2}$ then $p_2 \in \frac{1}{2}\mathbb{Z}\setminus\mathbb{Z}$. Furthermore, if $p_1 = \frac{1}{2}$ and $q_3 \in \mathbb{Z}$ then $p_3 \in \frac{1}{2}\mathbb{Z}\setminus\mathbb{Z}$.

A twistoid on the dicosm with the parameters described above will be denoted by $[\mathcal{DI} \mid c,p_1,p_2,p_3,q_3]$.  Note that such a twistoid contains $4cq_3(p_2-p_1)$ cubes, and therefore $192cq_3(p_2-p_1)$ flags.


The component of identity of $Sym([\mathcal{DI} \mid c,p_1,p_2,p_3,q_3])$ consists of all translations by a vector $(0,0,k)$ with $k\in \mathbb{Z}$. Therefore, there are precisely $2c$ different automorphisms in the component of identify of $Aut([\mathcal{DI} \mid c,p_1,p_2,p_3,q_3])$.  The rigidly isometric subgroup of $Aff(\mathcal{DI})$ is generated by translations $\alpha$ and $\beta$, mapping the axis of $\sigma_1$ to that of $\sigma_2$ and $\sigma_3$, respectively, together with the reflection with respect to a plane perpendicular to these axes.  Since the reflection $\rho$ with respect to a plane $z=k$ for $k \in \frac{1}{2}\mathbb{Z}$ is a symmetry of $[\mathcal{DI} \mid c,p_1,p_2,p_3,q_3]$ for all possible parameters, the rigidly isometric part of $Sym([\mathcal{DI} \mid c,p_1,p_2,p_3,q_3])$ can be determined by calculating for which parameters $\alpha$, $\beta$, and $\alpha\beta$ preserve the cubic lattice.

We choose $\alpha$ to be the translation by the vector $(p_2-p_1,0,0)$ which is a symmetry of the twistoid whenever $p_2-p_1\in \mathbb{Z}$. Likewise, we choose $\beta$ to be the translation by the vector $(p_3-p_1,q_3,0)$, which is a symmetry whenever $p_3-p_1,q_3\in \mathbb{Z}$. Similarly, the translation $\alpha \beta$ is a symmetry whenever $p_3-p_2, q_3 \in \mathbb{Z}$.
Therefore the group of rigidly isometric automorphisms of $[\mathcal{DI} \mid c,p_1,p_2,p_3,q_3]$ is either $\langle \rho' \rangle$, $\langle  \rho', \alpha' \rangle$, $\langle  \rho',  \beta' \rangle$, $\langle  \rho', \alpha' \beta' \rangle$, or $\langle  \rho', \alpha', \beta' \rangle$, where $\rho'$, $\alpha'$, and $\beta'$ are the elements in the group of rigidly isometric automorphisms that are induced by $\rho$, $\alpha$, and $\beta$, respectively.

In order to complete the description of $Sym([\mathcal{DI} \mid c,p_1,p_2,p_3,q_3])$, it remains to determine all symmetries in the deformable part of the $Sym([\mathcal{DI} \mid c,p_1,p_2,p_3,q_3])$.

Let $\phi$ be an isometry in the deformable part of $Sym([\mathcal{DI} \mid c,p_1,p_2,p_3,q_3])$.  We may assume that $\phi$ preserves the projection of $\sigma_1$, as the group of deformable symmetries is a quotient by the rigidly isometric ones. Furthermore, since all elements in the deformable part preserve the lattice formed by the projection onto $\Pi$ of the axes of twists in $\mathcal{DI}$, in this plane $\phi$ acts like an element in the dihedral group of order 8, preserving this lattice while fixing the projection of $\sigma_1$. In this case, $\phi$ is a half-turn with respect to an axis parallel to the corresponding axis in the projection, fixing $\mathcal{U}$, or $\phi$ is a twist along the axis of $\sigma_1$.

Since the projection of $\mathcal{U}$ into the plane $\Pi$ yields a square grid, any symmetry of the twistoid in the deformable part must be either a half-turn with respect to an axis with direction vector $(1,0,0)$, $(0,1,0)$, $(1,1,0)$ or $(1,-1,0)$ preserving the lattice, or a power of a 4-fold-twist $\eta$ along the axis of $\sigma_1$. Note that the second power of $\eta$ is equivalent under the component of identity to $\sigma_1$. Consequently, the half-turns along the axes determined by $(1,0,0)$ and $(0,1,0)$ when they occur, are equivalent under the component of identity and $\mathcal{DI}$ itself.  Similarly, the half-turns along the axes $(1,1,0)$ and $(1,-1,0)$ are equivalent.  Therefore we only need to determine whether $\eta$ and the half-turns $\gamma_1$ and $\gamma_2$ with respect to the axes generated by $(1,0,0)$ and $(1,1,0)$, respectively, occur as symmetries of the twistoid.

An easy calculation shows that $\gamma_1$ is a symmetry of $[\mathcal{DI} \mid c,p_1,p_2,p_3,q_3]$ if and only if $p_1=p_3$ or $p_3=\frac{p_2+p_1}{2}$. Furthermore, it follows from \cite[Proposition 5]{TOROIDS} that $\gamma_2$ is a symmetry of $[\mathcal{DI} \mid c,p_1,p_2,p_3,q_3]$ if and only if
\begin{equation}\label{eqclass1}
\frac{p_3-p_1}{q_3} \quad, \quad \frac{p_2-p_1}{q_3} \quad, \quad \frac{q_3^2-(p_3-p_1)^2}{q_3(p_2-p_1)} \in \mathbb{Z},
\end{equation}
and that $\eta$ is a symmetry if and only if
\begin{equation}\label{eqclass2}
\frac{p_3-p_1}{q_3} \quad, \quad \frac{p_2-p_1}{q_3} \quad, \quad \frac{q_3^2+(p_3-p_1)^2}{q_3(p_2-p_1)} \in \mathbb{Z}.
\end{equation}

To complete the classification, for any set of parameters, we first determine the subgroup of rigidly isometric symmetries. Then, for each possible rigidly isometric subgroup, we determine the subgroup of deformable symmetries.  When $c \in \mathbb{Z}$, the previous analysis classified the twistoids $[\mathcal{DI} \mid c,p_1,p_2,p_3,q_3]$ according to their rigidly isometric and deformable subgroups of symmetries shown as the families listed in Table~\ref{t:ugly}.  We note here that the deformable subgroup is in fact determined by the symmetry type of the toroid given by the projections of the axes of the $\sigma_i$'s.

\begin{table}
\begin{tabular}{|c|c|c|c|c|c|} \hline
& 1 & 2 & $2_{0,2}$ & $2_1$ & 4 \\ \hline
$\langle \rho' \rangle$ & $p_1=0, p_2=\frac{5}{2}$ & $p_1=0, p_2=\frac{5}{2}$ & $p_1=0, p_2=\frac{5}{2}$ & $p_1=0, p_2=\frac{3}{2}$& $p_1=0, p_2=\frac{21}{2}$ \\
& $p_3=0, q_3 = \frac{5}{2}$ &  $p_3=1, q_3 = \frac{1}{2}$  &  $p_3=0, q_3 = \frac{3}{2}$ & $p_3=1, q_3 = \frac{1}{2}$ &  $p_3=2, q_3 = \frac{1}{2}$ \\ \hline
$\langle \rho', \alpha' \rangle$ & $p_1=0, p_2=3$ & $p_1=0, p_2=13$ & $p_1=0, p_2=5$ & $p_1=0, p_2=4$& $p_1=0, p_2=10$ \\
& $p_3=\frac{3}{2}, q_3 = \frac{3}{2}$ &  $p_3=\frac{5}{2}, q_3 = \frac{1}{2}$  &  $p_3=\frac{5}{2}, q_3 = 1$ & $p_3=\frac{3}{2}, q_3 = \frac{1}{2}$ &  $p_3=1, q_3 = \frac{1}{2}$ \\ \hline
$\langle \rho',  \beta' \rangle$ & None & None & $p_1=0, p_2=\frac{3}{2}$ & None & $p_1=0, p_2=\frac{11}{2}$ \\
&  &    &  $p_3=0, q_3 = 2$ &  &  $p_3=1, q_3 = 2$ \\ \hline
$\langle \rho', \alpha' \beta' \rangle$ & None & None  & None & None & $p_1=0, p_2=\frac{5}{2}$ \\
&  &    &  & &  $p_3=\frac{1}{2}, q_3 = 1$ \\ \hline
$\langle \rho', \alpha', \beta' \rangle$ & $p_1=0, p_2=2$ & $p_1=0, p_2=5$ & $p_1=0, p_2=6$ & $p_1=0, p_2=3$& $p_1=0, p_2=20$ \\
& $p_3=0, q_3 =2$ &  $p_3=2, q_3 =1$  &  $p_3=3, q_3 =1$ & $p_3=2, q_3 =1$ &  $p_3=2, q_3 = 1$ \\ \hline
\end{tabular}
\caption{The 18 families of twistoids on the dicosm, where the translational component of $\sigma_1$ is $(0,0,c)$.  The columns of the table give the deformable parts and the rows give the rigidly isometric parts.\label{t:ugly}}
\end{table}

From \cite{TOROIDS} we observe that there are 5 possible symmetry types of a maps on the torus of type $\{4,4\}$. Each of the columns of Table~\ref{t:ugly} correspond to one of these types, and each entry of the table gives an example of a twistoid with the symmetries given in each row and the projection of the axes given by the columns.

In what follows we explain why there are no twistoids in the families labelled ``None'' in Table~ \ref{t:ugly}. We start with three remarks whose proofs are straight forward and thus omitted.

\begin{remark}\label{triallemma2}
If the axes of the generating twists project to $\Pi$ into exactly two different types of points (vertices, centers of squares, midpoints of horizontal edges, or midpoints of vertical edges) and $\gamma_2$ is a symmetry of a twistoid, then the projections of the axes of the generators are either vertices and centers of squares, or they are midpionts of vertical and horizontal edges.\end{remark}

\begin{remark}\label{triallemma1}
If the axis of any of the generating twists has points with $y$ coordinate in $\frac{1}{2}\mathbb{Z} \setminus \mathbb{Z}$, then the axes of $\sigma_1$ and $\sigma_2$ were chosen to contain points with integer $y$ coordinate, and therefore the axes of $\sigma_3$ and $\sigma_4$ have $y$ coordinate in $\frac{1}{2}\mathbb{Z} \setminus \mathbb{Z}$.
\end{remark}

The two above remarks imply that there are no possible parameters such that the twistoid $[\mathcal{DI} \mid c,p_1,p_2,p_3,q_3]$ has rigidly isometric subgroup giving the columns in Table~\ref{t:ugly} labeled $1$ or $2_1$ while having a deformable subgroup given by $\langle \beta' \rangle$ or $\langle \alpha' \beta' \rangle$.

\begin{remark}\label{triallemma3}
If the axes of the generating twists project to $\Pi$ into exactly two types of different points (vertices, centers of squares, midpoints of horizontal edges, or midpoints of vertical edges), and $\eta$ is a symmetry of a twistoid, then $p_1 = 0$ and the projections of the axes of the generators are all vertices and centers of squares, since midpionts of vertical and horizontal edges are interchanged by $\eta$.
\end{remark}

Remarks \ref{triallemma2} and \ref{triallemma3} imply that  there are no possible parameters such that the twistoid $[\mathcal{DI} \mid c,p_1,p_2,p_3,q_3]$ has rigidly isometric subgroup giving the column in Table~\ref{t:ugly} labeled $2$ while having a deformable subgroup given by $\langle \beta' \rangle$ or $\langle \alpha' \beta' \rangle$.

Considering the conditions on the parameters, an easy algebraic argument shows that if the deformable part contains $\gamma_1$ and the rigidly isometric part contains $\alpha \beta$, then the rigidly isometric part also contains $\alpha$. To see this assume first that $p_3=p_1$ and $p_3-p_2 , q_3 \in \mathbb{Z}.$  Then clearly $p_2-p_1$ is an integer, implying that $\alpha$ is a symmetry of the twistoid.  On the other hand, if $p_3=\frac{p_1+p_2}{2}$ and $p_3-p_2 , q_3 \in \mathbb{Z}$, then $\frac{p_1-p_2}{2}\in \mathbb{Z}$ implying that $\alpha$ is again a symmetry of the twistoid.  This shows that there are no twistoids having rigidly isometric subgroup determined by the column labeled $2_{0,2}$ while having a deformable part given by $\langle \alpha' \beta' \rangle$.

We have now shown that there are eighteen non-empty families of twistoids $[\mathcal{DI} \mid c,p_1,p_2,p_3,q_3]$ with $c\in \mathbb{Z}$, and given a possible set of parameters for an example of each family. For each of these families we derive the number of flag-orbits of a twistoid $[\mathcal{DI} \mid c,p_1,p_2,p_3,q_3]$.

As noted above, the number of flags of the twistoid $[\mathcal{DI} \mid c,p_1,p_2,p_3,q_3]$ is $192cq_3(p_2-p_1)$. Since the component of identify of $Aut([\mathcal{DI} \mid c,p_1,p_2,p_3,q_3])$ has $2c$ different elements, and the reflection $\rho$ is always a symmetry, it follows that each such twistoid has at most $48q_3(p_2-p_1)$ flag-orbits. In fact, this number is achieved when the projection of the axes of the twists of $\mathcal{DI}$ form a lattice inducing a torus in class $4$ (where the deformable subgroup is trivial) and the rigidly isometric subgroup consists only of $\langle \rho' \rangle$.

If only one of $\alpha$, $\beta$ or $\alpha \beta$ is a symmetry of $[\mathcal{DI} \mid c,p_1,p_2,p_3,q_3]$ then the number of flag-orbits is reduced by a factor of $2$. Furthermore, if all three of them are symmetries, then the number of flag-orbits is reduced by a factor of $4$. The precise number of flag-orbits for a twistoid in each family is then determined by the deformable part. If the deformable part is a group of order $n$ then the number of flag-orbits is reduced by a factor of $n$.  For example, $[\mathcal{DI} \mid 28,0,2,0,2]$ has $3\cdot 2 \cdot 2 = 12$ flag-orbits.

\medskip
{\bf Case 2:} The axes of the generators have direction vector $(1,-1,0)$.
\medskip

Let $\Pi$ be the plane $x=y$.  We denote by $(p_1, p_1, q_1)$, $(p_2, p_2, q_2)$ and $(p_3, p_3, q_3)$ the projections of $\sigma_1$, $\sigma_2$, and $\sigma_3$ to $\Pi$, respectively. Up to duality, we may choose $p_1 \in \{0, \frac{1}{4}\}$ and $q_1=0$. Furthermore, by Lemma \ref{q2} we may assume also that $q_2=0$. Note that $p_2, p_3 \in \frac{1}
{4}\mathbb{Z}$ and $q_3 \in \frac{1}{2}\mathbb{Z}$.

The norm of the translational component $c$ of $\sigma_1$ is an integer multiple of $\sqrt{2}$ if $p_1 = 0$,
or an odd integer multiple of $\frac{\sqrt{2}}{2}$
if $p_1 = \frac{1}{4}$. This implies that either $p_1, p_2, p_3$ are all in $\frac{1}{2}\mathbb{Z}$ or none of them are.

A twistoid on the dicosm with the parameters described above will be denoted by $[\mathcal{DI} \mid c,p_1,p_2,p_3,q_3]$. Note that these twistoids are differentiated from those in the previous case by whether $c \in \mathbb{Z}$ or not.  The twistoid $[\mathcal{DI} \mid c,p_1,p_2,p_3,q_3]$ has $4c\sqrt{2}q_3(p_2-p_1)$ cubes, and therefore $192c\sqrt{2}q_3(p_2-p_1)$ flags.


The component of identity of $Sym([\mathcal{DI} \mid c,p_1,p_2,p_3,q_3])$ consists of all translations by a vector $(k,-k,0)$ with $k\in \mathbb{Z}$. That is, there are precisely $\sqrt{2}c$ different elements of the component of identity.

As before, the group of rigidly isometric part of $Aff(\mathcal{DI})$ is generated by translations $\alpha$, $\beta$, and $\alpha \beta$ mapping the axis of $\sigma_1$ to that of $\sigma_2$, $\sigma_3$ and $\sigma_4$, respectively, together with the reflection with respect to a plane perpendicular to these axes.

The reflection with respect to a plane through the origin with normal vector $(1,-1,0)$ is a symmetry of $[\mathcal{DI} \mid c,p_1,p_2,p_3,q_3]$ for all possible parameters.

If $p_2-p_1 \in \mathbb{Z}$ then the translation $\alpha$ by the vector $(p_2-p_1, p_2-p_1, 0)$ also preserves $\mathcal{U}$ and normalizes $\mathcal{DI}$, and therefore is a symmetry of $[\mathcal{DI} \mid c,p_1,p_2,p_3,q_3]$ for all possible parameters.
On the other hand, if $p_2-p_1 \in \frac{\mathbb{Z}}{2} \setminus \mathbb{Z}$ then we may take $\alpha$ to be the translation by the vector $(q_2-q_1+\frac{1}{2}, q_2-q_1-\frac{1}{2},0)$; in this case $\alpha$ does not preserve $\Pi$. It follows that the translation $\beta$ is a symmetry if and only if $\alpha \beta$ is a symmetry.

Since the projection of $\mathcal{U}$ into the plane $\Pi$ yields a rectangular grid, by arguments similar to the previous case, any nontrivial symmetry of the twistoid in the deformable part must be either a half-turn about an axis with direction vector $(1,1,0)$ or $(0,0,1)$ preserving the lattice, or a 2-fold-twist along the axis of $\sigma_1$. Note that the latter is equivalent under the component of identity to $\sigma_1$. Consequently, the other two half-turns, in the case where they occur, are equivalent under the component of identity and $\mathcal{DI}$ itself.

Therefore, in order to determine the entire set of symmetries of $[\mathcal{DI} \mid c,p_1,p_2,p_3,q_3]$ it remains to determine for which parameters the translation $\beta$ by $(p_3-p_1,p_3-p_1,q_3-q_1)$ and the half-turn $\chi$ along the axis $\{k(1,1,0) \mid k \in \mathbb{R}\}$ are symmetries of the twistoid.  An easy calculation shows that $\chi$ is a symmetry of $[\mathcal{DI} \mid c,p_1,p_2,p_3,q_3]$ if and only if $p_3=p_1$ or $p_3 = \frac{p_1+p_2}{2}$. On the other hand, $\beta$ is a symmetry if and only if $q_3 \in \mathbb{Z}$.

Recall that $p_2-p_1 \in {\mathbb{Z}}$ (resp. $p_3-p_1 \in \mathbb{Z}$) if and only if the direction vector of $\alpha$ (resp. $\beta$) can be chosen to be perpendicular to $(1,-1,0)$.

The previous analysis classifies the twistoids $[\mathcal{DI} \mid c,p_1,p_2,p_3,q_3]$ in the following nonempty families when $c \in \frac{\sqrt{2}}{2}\mathbb{Z}$:
\begin{itemize}
  \item $q_3 \in \mathbb{Z}$ and ($p_3 = p_1$ or $p_3 = \frac{p_1+p_2}{2}$), admitting all possible symmetries arising from the rigid as well as the deformable part. Twistoids in this family have $12q_3(p_2-p_1)$ flag-orbits.
  \item $q_3 \notin \mathbb{Z}$ and ($p_3 = p_1$ or $p_3 = \frac{p_1+p_2}{2}$), admitting all possible symmetries arising from the deformable part, but not from the rigid part. Twistoids in this family have $24q_3(p_2-p_1)$ flag-orbits.
  \item $q_3 \in \mathbb{Z}$, $p_3\ne p_1$, and $p_3\ne \frac{p_1+p_2}{2}$, admitting all symmetries arising from the rigid part, but not from the deformable part. Twistoids in this family have $24q_3(p_2-p_1)$ flag-orbits.
  \item $q_3 \notin \mathbb{Z}$, $p_3\ne p_1$, and $p_3\ne \frac{p_1+p_2}{2}$, including all twistoids not considered in the previous families. Twistoids in this family have $48q_3(p_2-p_1)$ flag-orbits.
\end{itemize}

\section{Tricosm}\label{s:tri}

In this section we will show that the twistoids on tricosm can be classified into three distinct families.

It follows from the notation in~\cite{ConwaySOT} that the group $\mathcal{TR}$ is generated by three twists $\sigma_1$, $\sigma_2$, and $\sigma_3$, whose axes are parallel and project orthogonally to the vertices of an equilateral triangle. Furthermore, their translational component coincides. It can be easily seen that $\sigma_1$ and $\sigma_2$ determine $\sigma_3$, for example by $\sigma_3^2 = \sigma_1 \sigma_2$.
Hence, the group $\mathcal{TR}$ is generated by two twists whose axes are parallel, having equal translational component.

Since the generators are 3-fold twists, they must be parallel to the diagonal of a cube of the tessellation, that is, of types $VIII$ or $IX$. Without loss of generality we assume that the translational component of $\sigma_1$ and $\sigma_2$ is $c(1,1,1)/\sqrt{3}$. More generally, for $i\in \{1,2\}$, the axis of $\sigma_i$ can be taken to be $\{t(1,1,1) + (p_i, q_i, r_i) \mid t \in \mathbb{R}\}$.

Let $\Pi$ be the plane through $(0,0,0)$ with normal vector $(1,1,1)$. Then the projection of $\mathcal{U} $ to $\Pi$ induces the vertex lattice $\Lambda$ of a tessellation of $\Pi$ by equilateral triangles. Since $\sigma_i$ preserves $\mathcal{U}$, it induces a 3-fold rotation on $\Lambda$
with respect either to a point of $\Lambda$ or to a centroid of an equilateral triangle (marked with a gray `X' in Figure~\ref{fig:LTR}, where the projection of the three coordinate axes are marked). Without loss of generality we may assume that $(p_1,q_1,r_1) = (0,0,0)$ if the axis of $\sigma_1$ contains points of $\Lambda$, or $(p_1,q_1,r_1) = \frac{1}{3}(1,-1,0)$ otherwise.
Note that in the latter case, the axis of $\sigma_1$ is the axis of a Petrie polygon of $\mathcal{U}$.

\begin{figure}
\begin{center}
\includegraphics[width=5cm, height=5cm]{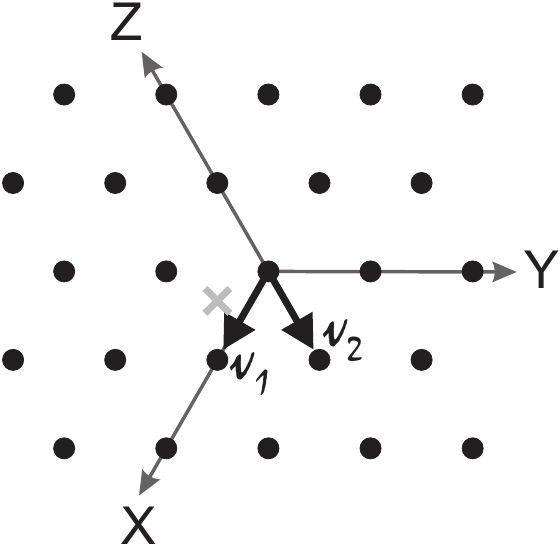}
\caption{Projection $\Lambda$ of $\mathcal{U}$ to $\Pi$ \label{fig:LTR}}
\end{center}
\end{figure}

If the axis of a $3$-fold twist preserving $\mathcal{U}$ contains a vertex of $\mathcal{U}$ then $c \in \sqrt{3}\mathbb{Z}$. On the other hand, if the axis of a $3$-fold twist is the axis of a Petrie polygon of $\mathcal{U}$ then $c \in \frac{\sqrt{3}}{3}\mathbb{Z} \setminus \sqrt{3}\mathbb{Z}$. It follows that if the axis of $\sigma_1$ contains a vertex of $\mathcal{U}$ then the axis of $\sigma_2$ also contains a vertex of $\mathcal{U}$.

Additionally, if $\frac{c}{\sqrt{3}} = 3k+1$ for some integer $k$, then the axis of $\sigma_1$ is the axis of a right Petrie polygon of $\mathcal{U}$ and induces a clockwise 3-fold rotation in $\Pi$.  On the other hand if $\frac{c}{\sqrt{3}} = 3k+2$, then the axis of $\sigma_1$ is the axis of a left Petrie polygon of $\mathcal{U}$ and induces a counterclockwise 3-fold rotation. In either case, the projections of the axes of $\sigma_i$ all are centers of triangles of the same color (in the natural 2-coloring of the tessellation of the plane by regular triangles).

As in Section~\ref{s:di}, the value of $c$ determines the location of the axis of $\sigma_1$. If $c \in \sqrt{3}\mathbb{Z}$ then $(p_1,q_1,r_1) = (0,0,0)$, otherwise $(p_1,q_1,r_1) = \frac{1}{3}(1,-1,0)$. For convenience, we locate the axis of $\sigma_2$ in terms of its projection to $\Pi$, with respect to the basis $B=\{v_1,v_2\}$ where $v_1=\frac{1}{3}(2,-1,-1)$ and $v_2 = \frac{1}{3}(1,1,-2)$ (see Figure~\ref{fig:LTR}).

In this notation, the intersection of the axis of $\sigma_2$ with $\Pi$ is described as $(a,b) = av_1 +bv_2$ with $a, b \in \mathbb{Z}$.  We may assume without loss of generality that both $a$ and $b$ are positive. Since $\sigma_3$ was chosen such that $\sigma_1 \sigma_2 = \sigma_3^2$, the intersection of the axes of $\sigma_1$, $\sigma_2$, and $\sigma_3$ with $\Pi$ forms an equilateral triangle such that a clockwise $3$-fold rotation permutes the intersections with $\Pi$ of the axes of $\sigma_1$, $\sigma_3$ and $\sigma_2$ in that order.

The parameters $c,a,b$ then determine completely the twistoid, and therefore a twistoid with these parameters is denoted by $[\mathcal{TR} \mid c,a,b]$.

Using that $|av_1+bv_2| = \sqrt{\frac{2a^2+2b^2+2ab}{3}}$ it follows that the volume (number of cubes) of the fundamental region with these parameters is $\sqrt{3}c(2a^2+2b^2+2ab)$ and therefore the twistoid $[\mathcal{TR} \mid c,a,b]$ has $48c\sqrt{3}(a^2+b^2+ab)$ flags.

The component of identity of $Sym([\mathcal{TR} \mid c,a,b])$ consists of all translations by a vector $(k,k,k)$ with $k\in \mathbb{Z}$. Therefore, there are precisely $\sqrt{3}c$ different elements in the component of identity of $Aut([\mathcal{TR} \mid c,a,b])$. Since the deformable part in this case is trivial, it only remains to determine the rigidly isometric symmetries.

Define $\alpha$ as a translation sending the axis of $\sigma_1$ to the axis of $\sigma_2$ while preserving the cubic lattice  Then $\alpha$ is a translation by the vector $a v_1 + b v_2 + k (1,1,1)$ for some $k \in \frac{1}{3} \mathbb{Z}$.  Without loss of generality we can assume that $k \in \{0,\frac{1}{3}, \frac{2}{3} \}$.
Considering the natural 2-coloring of the triangles in $\Lambda$, such an $\alpha$ always exists, and yields a non-trivial rigidly isometric part.
Note that a similar translation sending the axis of $\sigma_1$ to the axis of $\sigma_3$ is equivalent to $\alpha$ under under the quotient of $\mathcal{TR}$; in fact it is the conjugate of $\alpha$ by $\sigma_1^2$.

If the parameters allow it, we define $\chi$ as the half-turn about a line meeting the axes of $\sigma_1$ and $\sigma_2$, perpendicular to both and preserving the cubic tessellation.
The projection of the axis $\chi$ onto the plane $\Pi$ gives a line, about  $\chi$ which acts like a reflection when restricted to $\Lambda$.
 When $c \notin \sqrt{3} \mathbb{Z}$, the only possible way for the axis for
$\chi$ to project to such a line is if $a=b$.
Furthermore, in this case, it can be seen that $\chi$ is a symmetry of the twistoid $[\mathcal{TR} \mid c,a,b]$ if and only if $a = b$. If $c \in \mathbb{Z}$ then, in order to preserve $\Lambda$, the axis of $\chi$ could also project to the edges of the triangles; however, no axis of a $2$-fold rotation preserving the cubic tessellation projects to such lines. It is easy to see now that, for all $c$, $\chi$ is a symmetry of $[\mathcal{TR} \mid c,a,b]$ if and only if $a=b$.

In case there is no half-turn about lines meeting the axes of $\sigma_1$ and $\sigma_2$, perpendicular to both and preserving the tessellation, there may be a half-turn $\zeta$ about a line perpendicular to the axis of $\sigma_1$, interchanging the axes of $\sigma_2$ and $\sigma_3$. By arguments similar as above, we can easily see that $[\mathcal{TR} \mid c,a,b]$ has $\zeta$ as a symmetry if and only if $a=0$ or $b=0$.

Finally, we note that there is a $6$-fold rotation about the axis of $\sigma_1$ preserving the group; however, such an isometry does not preserve the cubic tessellation, implying that such an outer automorphism of $\mathcal{TR}$ can never be realized as a symmetry of a twistoid.

Recall that the outer automorphism group of $\mathcal{TR}$ is isomorphic to $\mathbb{Z}_2 \times S_3$. Since the symmetry group of every twistoid $[\mathcal{TR} \mid c,a,b]$ contains $\alpha$, which projects to an element $\overline{\alpha}$ of order $3$ in $Out(\mathcal{TR})$, it follows that $\langle \overline{\alpha} \rangle$ is normal in $Out(\mathcal{TR})$ and that $Out(\mathcal{TR})/\langle \overline{\alpha} \rangle$ is isomorphic to $\mathbb{Z}_2\times \mathbb{Z}_2$. Furthermore, the non-trivial elements of the latter are the classes of $\chi$, $\zeta$ and $\chi\zeta$. Note that $\chi\zeta$ is a $6$-fold twist, and no such isometries preserve $\mathcal{U}$. Consequently, at most one of $\chi$ and $\zeta$ is a symmetry of the twistoid. Therefore there are three non-empty families of twistoids in the tricosm as explained next.

\begin{itemize}
  \item $a=b$, allowing the symmetry $\chi$. Twistoids in this family have $8(a^2+b^2+ab)$ flag-orbits.
  \item $a=0$ or $b=0$, allowing the symmetry $\zeta$. Twistoids in this family have $8(a^2+b^2+ab)$ flag-orbits.
  \item $ab(a-b)\ne 0$, allowing none of $\chi$ or $\zeta$. Twistoids in this family have $16(a^2+b^2+ab)$ flag-orbits.
\end{itemize}

\section{Tetracosm}\label{s:tetra}

In this section we will show that the twistoids on tetracosm can be classified into four distinct families.

The tetracosm is the quotient of $\mathbb{E}^3$ by the group $\mathcal{TE}$.  The canonical generators of this group are two 4-fold twists (which we denote by $\sigma$ and $\sigma_1$) and one 2-fold twist (which we denote by $\sigma_2$), all of which have parallel axes.
The translational component of $\sigma$ and $\sigma_1$ is the same and we denote its norm by $c$.  The norm of the translational component of $\sigma_2$ is then $2c$.
It can be easily seen that $\sigma$ can be obtained from $\sigma_1$ and $\sigma_2$.
Hence, the group $\mathcal{TR}$ is generated by two twists whose axes are parallel; one of them is a half-turn twist with translational component twice the translational component of the other one, which is a $4$-fold twist.

Since one of the generators is a 4-fold twist, all the axes of the generators must be parallel to some coordinate axis, that is, of types $I$, $II$ or $III$. Without loss of generality we assume that the translational component of $\sigma_1$  is $(0,0,c)$. More generally, for $i\in \{1,2\}$, the axis of $\sigma_i$ is $\{t(0,0,1) + (p_i, q_i, 0) \mid t \in \mathbb{R}\}$. By choosing $\sigma$ such that $\sigma_1^{-1} \sigma_2 = \sigma$ we observe that the axis of $\sigma$ is $\{t(0,0,1) + (p_2-q_2, p_2+q_2, 0) \mid t \in \mathbb{R}\}$.

Let $\Pi$ be the $xy$-plane. Then the projection of $\mathcal{U}$ to $\Pi$ induces a square tessellation $\Lambda$. Since $\sigma_1$ preserves $\mathcal{U}$, it acts like a 4-fold rotation in $\Lambda$ with respect either to a vertex of $\Lambda$ or to a centroid of a square. Without loss of generality and up to duality, we may assume that $(p_1,q_1) = (0,0)$. Furthermore, $p_2, q_2 \in \frac{1}{2}\mathbb{Z}$. For convenience we shall denote $p_2$ and $q_2$ by $p$ and $q$ respectively.

Since $\sigma_1^{-1} \sigma_2 = \sigma$, the intersection of the axis of $\sigma$ with $\Pi$ forms a right isosceles triangle with the intersection of the axes of $\sigma_1$ and $\sigma_2$ with $\Pi$ such that the right angle is at the projection of $\sigma_2$.
The parameters $c,p,q$ then determine completely the twistoid and therefore a twistoid with these parameters is denoted by $[\mathcal{TE} \mid c,p,q]$.  It is easy to see that the volume (number of cubes) of the fundamental region with these parameters is $4c(p^2+q^2)$ and therefore the twistoid $[\mathcal{TE} \mid c,p,q]$ has $192c(p^2+q^2)$ flags.

The component of identity of $Sym([\mathcal{TE} \mid c,p,q])$ consists of all translations by a vector $(0,0,k)$ with $k\in \mathbb{Z}$. Note that there are $4c$ different elements of the component of identity of $Aut([\mathcal{TE} \mid c,p,q])$.  Since the deformable part in this case is again trivial, it only remains to determine all symmetries of the twistoid that project to the rigidly isometric part.

We define $\alpha$ as the translation by the vector $(p-q,p+q)$ sending the axis of $\sigma_1$ to the axis of $\sigma$. Then $\alpha$ is a symmetry of the twistoid $[\mathcal{TE} \mid c,p,q]$ if and only if $p-q \in \mathbb{Z}$.  Furthermore, we define $\chi$ as the half-turn about the line $\{t(p,q,0) \mid t \in \mathbb{R}\}$. The axis $\chi$ is a line, about which $\chi$ acts like a reflection when considering its action on $\Lambda$. It is easy to see that $\chi$ is a symmetry of $[\mathcal{TE} \mid c,p,q]$ if and only if $pq(p-q)=0$.

Recall that the outer automorphism group of $\mathcal{TE}$ is isomorphic to $\mathbb{Z}_2^2$. Note that $\alpha$ and $\chi$ represent different outer automorphisms. Furthermore, $\alpha \chi$ is a symmetry of the twistoid $[\mathcal{TE} \mid c,p,q]$ only if $\alpha$ and $\chi$ are also symmetries. (Note here that $\alpha \chi$ is a half-turn twist whose axis must act like a glide reflection in $\Lambda$, which can happen only if $pq(p-q)=0$. Furthermore, it maps the axis of $\sigma_1$ to a conjugate of the axis of $\sigma$, which implies that $p-q \in \mathbb{Z}$.) Therefore, there are four non-empty families of twistoids in the tetracosm as follows.

\begin{itemize}
  \item $pq(p-q)=0$ and $p-q \in \mathbb{Z}$, allowing symmetries $\alpha$ and $\chi$. Twistoids in this family have $12(p^2+q^2)$ flag-orbits.
  \item $pq(p-q)\ne 0$ and $p-q \in \mathbb{Z}$, allowing symmetry $\alpha$ but no $\chi$. Twistoids in this family have $24(p^2+q^2)$ flag-orbits.
  \item $pq(p-q)=0$ and $p-q \notin \mathbb{Z}$, allowing symmetry $\chi$ but not $\alpha$. Twistoids in this family have $24(p^2+q^2)$ flag-orbits.
  \item $pq(p-q) \ne 0$ and $p-q \notin \mathbb{Z}$, not allowing symmetries $\alpha$ or $\chi$. Twistoids in this family have $48(p^2+q^2)$ flag-orbits.
\end{itemize}

\section{Toroidal Covers}
\label{s:toroids}

As noted in the introduction, the classification of cubic twistoids on the 3-torus (also called equivelar 4-toroids) by symmetry type was completed in~\cite{TOROIDS}.  It is natural to relate this classification to our classifications in Sections~\ref{s:di}, \ref{s:tri}, and \ref{s:tetra}.

When describing the minimal toroidal covers of the twistoid $\mathcal{T} = [\mathcal{DI} \mid c,p_1,p_2,p_3,q_3]$ on the dicosm we first consider when $c\in\mathbb{Z}$, that is, the axes of the standard generators are parallel to the $z$-axis.  Given such a twistoid on the dicosm $\mathcal{T} = [\mathcal{DI} \mid c,p_1,p_2,p_3,q_3] = \mathcal{U} / \mathcal{TR}$, the minimal covering 4-toroid $\mathcal{P}=\mathcal{U} / \Lambda$ can be described by the three generating translations of $\Lambda$ by the vectors $t_1=(0,0,2c)$, $t_2=(2(p_2-p_1),0,0)$, and $t_3=2(p_3-p_1,2q_3,0)$.

We note that all these toroids arise from a vertical translation lattice as defined in Section 6.2 of~\cite{TOROIDS}. Furthermore, the minimal toroidal cover of all twistoids in each of the families described in Table~\ref{t:ugly} depends only on the column and not on the row where the family is located. From Table 2 in~\cite{TOROIDS} it follows that a twistoid in a family appearing in the first column has a minimal toroidal cover $\mathcal{P}$ in class 1 whenever both $p_3=p_1$ and $p_2-p_1=q_3=c$;  the cover is in class 3 otherwise. If the family of $[\mathcal{DI} \mid c,p_1,p_2,p_3,q_3]$ appears in the column labelled $2$, then $\mathcal{P}$ is in class $6_C$. If the family of $[\mathcal{DI} \mid c,p_1,p_2,p_3,q_3]$ appears in the column labelled $2_{0,2}$ then $\mathcal{P}$ is in class $6_A$. If  the family of $[\mathcal{DI} \mid c,p_1,p_2,p_3,q_3]$ appears in the column labelled $2_1$ then $\mathcal{P}$ is in class $6_B$. Finally, if  the family of $[\mathcal{DI} \mid c,p_1,p_2,p_3,q_3]$ appears in the column labelled $4$ then $\mathcal{P}$ is in class $12_A$. This classification implies that the automorphism group of $\mathcal{P}$ acts on its flags with $1,3,6,$ or $12$ orbits.

We now turn our attention to twistoids $[\mathcal{DI} \mid c,p_1,p_2,p_3,q_3]$ in the dicosm with $c\in \frac{\sqrt{2}}{2}\mathbb{Z}$.  Given such a twistoid on the dicosm $\mathcal{T} = [\mathcal{DI} \mid c,p_1,p_2,p_3,q_3] = \mathcal{U} / \mathcal{DI}$, the minimal covering 4-toroid $\mathcal{P}=\mathcal{U} / \Lambda$ can be described by the three generating translations of $\Lambda$ by the vectors $t_1=\sqrt{2}c(1,-1,0)$, $t_2=2(p_2-p_1,p_2-p_1,0)$, and $t_3=2(p_3-p_1,p_3-p_1,q_3)$.  In order to understand these minimal toroidal covers more easily we define the following parameters (as in the notation of~\cite{TOROIDS}).  Let $a:=2(p_3-p_1)$, $b:=q_3$, $c':=2(p_2-p_1)$, and $d:=\sqrt{2}c$.

Recall that there is a symmetry $\gamma_1$ of the twistoid $[\mathcal{DI} \mid c,p_1,p_2,p_3,q_3]$ if and only if either $p_3=p_1$ or $p_3= \frac{p_2+p_1}{2}$.  If $p_3 = p_1$, then $\mathcal{T}$ is generated by the translations $(a,a,0)$, $(d,-d,0)$, and $(0,0,b)$, which is in class $3$ or $6_B$ as in Table 2 of~\cite{TOROIDS}.  On the other hand, if $p_3= \frac{p_2+p_1}{2}$, then $T$ is generated by the translations $(2a,2a,0)$, $(d,-d,0)$, and $(a,a,b)$, which is in class $6_B$ as in Table 4 of~\cite{TOROIDS}.  If $\gamma_1$ is not a symmetry, then $c$ does not divide $2a$, and thus according to Line 2 of Table 6 of~\cite{TOROIDS}, $\mathcal{P}$ is in class $12_B$.  This classification implies that the automorphism group of $\mathcal{P}$ acts on its flags with $3$, $6$ or $12$ orbits.

We next consider toroidal covers of twistoids on the tricosm.  Given such a twistoid $\mathcal{T} = [\mathcal{TR} \mid c,a,b] = \mathcal{U} / \mathcal{TR}$, the minimal covering 4-toroid $\mathcal{P}=\mathcal{U} / \Lambda$ can be described by the three generating translations of $\Lambda$ by the vectors $t_1=c\sqrt{3}(1,1,1)$, $ t_2=(a,b,-a-b)$, and $t_3=(-b,a+b,-a)$.  The last two vectors are determined by finding the smallest translations that send the axis of $\sigma_1$ to the axis of one of its conjugates.  Note that there is a fundamental region for the group $\mathcal{TR}$, different from the one we described, which is a prism over a rhombus, where the rhombus is determined by the two vectors $t_2$ and $t_3$.

There are three distinct cases for these minimal toroidal covers.  First, if $ab=0$, then the toroid $\mathcal{P}$ is described in Line 1 of Table 5 of~\cite{TOROIDS}, and is in class $4$.  Next, if $a-b=0$, then the toroid $\mathcal{P}$ is described in Line 3 of Table 5 of~\cite{TOROIDS}, and is also in class $4$. Note that the translation group described in Line 3 is the same as the one we described, but with different generators.   Finally if $pq(p-q)\ne 0$ then the toroid $\mathcal{P}$ is described in Line 5 of Table 5 of~\cite{TOROIDS}, and is in class $8$.

In each case, the lattice associated with the group $\Lambda$ is invariant under the 3-fold rotation about the line generated by the vector $(1,1,1)$ - that symmetry is denoted by $R_2R_1$ in~\cite{TOROIDS}.   Furthermore, the classification of such 4-toroids implies that the automorphism group of $\mathcal{P}$ acts on its flags with $4$ or $8$ orbits.

Finally, given a twistoid on the tetracosm $\mathcal{T} = [\mathcal{TE} \mid c,p,q] = \mathcal{U} / \mathcal{TE}$, the minimal covering 4-toroid $\mathcal{P}=\mathcal{U} / \Lambda$ can be described by the three generating translations of $\Lambda$ by the vectors $t_1=(0,0,4c)$, $ t_2=(2p,2q,0)$, and $t_3=(2q,-2p,0)$.

We observe that there are two distinct cases for these minimal toroidal covers; either $pq(p-q)=0$ when $\mathcal{P}$ is in class $3$, or $pq(p-q)\ne 0$ when $\mathcal{P}$ is in class $6_C$.  In each case, the lattice associated with the group $\Lambda$ is called a \textit{vertical translation lattice} in~\cite{TOROIDS}, and is invariant under the reflection in the plane $z=0$.  Furthermore, the classification of such 4-toroids implies that the automorphism group of $\mathcal{P}$ acts on its flags with $3,$ or $6$ orbits (see Table 2 of~\cite{TOROIDS}).

Finally, we point out that the minimal toroidal cover of an oriented twistoid (on a manifold other than the torus) never has 24 flag orbits.

\bibliographystyle{amsplain}

\begin{thebibliography}{9}

\bibitem{BK} U. Brehm, W.K\"uhnel, Equivelar maps on the torus, {\em European J. Combin.} {\bf 29} (2008), 1843--1861.

\bibitem{Charlap}
L. S. Charlap.  {\em Bieberbach groups and flat manifolds.} Universitext Mathematical Physics and Mathematics. Springer, New York, 1986.




\bibitem{ConwaySOT}
J. H. Conway, H. Burgiel, C. Goodman-Strauss,
{\em The symmetries of things}, A K Peters Ltd., Wellesley, MA, 2008.

\bibitem{CR1} J. H. Conway, J. P. Rosetti, Describing the Platycosms, arXiv:math.DG/0311476.

\bibitem{CR2} J. H. Conway, J. P. Rossetti, Hearing the platycosms, {\em Math. Res. Lett.} {\bf 13} (2006), 475--494.

\bibitem{Coxeter48} H. S. M. Coxeter, Configurations and maps, {\em Rep. Math. Colloq. (2)} {\bf 8} (1948), 18--38.

\bibitem{CM} H. S. M. Coxeter, W. O. J. Moser, {\em Generators and Relations for Discrete Groups, 4th edition}, Springer-Verlag, Berlin, 1984.

\bibitem{DR} P. G. Doyle, J. P. Rossetti, Tetra and Didi, the cosmic spectral twins, {\em Geometry and Topology} {\bf 8} (2004), 1227--1242.

\bibitem{Edmonds} A. L. Edmonds, J. H. Ewing, R. S. Kulkarni.  Regular tessellations of surfaces and $(p,q,2)$-triangle groups, {\em Ann. of Math (2)} {\bf 116} (1982), 113--132.

\bibitem{Hillman}
Hillman. Flat 4-manifold groups.  New Zealand J. Math. 24 (1995), no. 1, 29-40.

\bibitem{HW} W. Hantzsche; H. Wendt, Dreidimensionale euklidische Raumformen. {\em Math. Annalen.} {\bf 110} (1934--35), 593--611.

\bibitem{HMS} M. I. Hartley, P. McMullen, E. Schulte, Symmetric tessellations on Euclidean space-forms, {\em Canad. J. Math.} {\bf 51} (1999), 1230--1239.

\bibitem{TOROIDS} I. Hubard, A. Orbanic, D. Pellicer, A. I. Weiss, Symmetries of equivelar 4-toroids, {\em Discrete Comput. Geom.} {\bf 48} (2012), 1110–-1136.

\bibitem{TWISTOIDS1} I. Hubard, M. Mixer, D. Pellicer, A. I. Weiss, Cubic tessellations of the Didicosm, in press.

\bibitem{Isangulov}
R. R. Isangulov, Isospectral flat 3-manifolds, {\em Siberian Mathematical Journal} {\bf 45} (2004), 894--914.

\bibitem{Kurth} W. Kurth, Enumeration of Platonic maps on the torus, {\em Discrete Math.} {\bf 61} (1986), 71--83.

\bibitem{MSHigherToroidal} P. McMullen, E. Schulte, Higher toroidal regular polytopes, {\em Adv. Math.} {\bf 117} (1996), 17--51.

\bibitem{ARP}
\newblock P. McMullen, E. Schulte,
\newblock {\em Abstract regular polytopes},
\newblock {\em Encyclopedia of Math. And its Applic.} {\bf 92}, Cambridge, 2002.



\bibitem{Nowacki} W. Nowacki, Die euklidischen, dreidimensionalen, geschlossenen und offenen Raumformen, {\em Comment. Math. Helv.} {\bf 7} (1934), 81--93.

\bibitem{Thomassen} C. Thomassen,  Tilings of the torus and the Klein bottle and vertex-transitive graphs on a fixed surface, {\em Trans. Amer. Math. Soc.} {\bf 323} (1991), 605--635.

\bibitem{unclesteve} S. Wilson, Uniform maps on the Klein bottle, {\em J. Geom. Graph.} {\bf 10} (2006), 161--171.

\bibitem{Wolf} J. A. Wolf, {\em Spaces of constant curvature} (5th edition), Publish or Perish, Wilmington, DE, 1984.


\end{thebibliography}

\end{document}